\documentclass[letterpaper, 10 pt, conference]{ieeeconf}  % Comment this line out
\IEEEoverridecommandlockouts                              % This command is only
\usepackage{mathtools}
\usepackage{amsfonts}
\usepackage{booktabs}
\usepackage{cite}
\usepackage[usenames,dvipsnames]{xcolor}
\usepackage{mathtools}
\usepackage{flushend}
\newcommand{\N}{\mathbb{N}}
\newcommand{\R}{\mathbb{R}}
\newcommand{\E}{\mathbb{E}}
\newcommand{\C}{\mathbb{C}}

\usepackage{comment}

\newif\ifshow
\showtrue
\ifshow

\else
\excludecomment{scratch}
\fi

\newtheorem{remark}{\bfseries Remark}
\newtheorem{example}{\bfseries Example}

\newtheorem{theorem}{\bfseries Theorem}

\usepackage{xargs}                      % Use more than one optional parameter in a new commands
\usepackage[colorinlistoftodos,prependcaption,textsize=tiny]{todonotes}
\newcommandx{\unsure}[2][1=]{\todo[linecolor=red,backgroundcolor=red!25,bordercolor=red,#1]{#2}}
\newcommandx{\change}[2][1=]{\todo[linecolor=blue,backgroundcolor=blue!25,bordercolor=blue,#1]{#2}}
\newcommandx{\info}[2][1=]{\todo[linecolor=OliveGreen,backgroundcolor=OliveGreen!25,bordercolor=OliveGreen,#1]{#2}}
\newcommandx{\improvement}[2][1=]{\todo[linecolor=Plum,backgroundcolor=Plum!25,bordercolor=Plum,#1]{#2}}
\newcommandx{\thiswillnotshow}[2][1=]{\todo[disable,#1]{#2}}

\title{\LARGE \bf
Estimating stationary characteristic functions of stochastic systems \\ via semidefinite programming}

\author{Khem Raj Ghusinga$^{1}$, Andrew Lamperski$^{2}$, and Abhyudai Singh$^{1}$% <-this % stops a space
\thanks{$^{1}$ Khem Raj Ghusinga and Abhyudai Singh are with the Department of Electrical and Computer Engineering, University of Delaware, Newark, DE, USA 19716.
        {\tt\small \{khem,absingh\}@udel.edu}}%
\thanks{$^{2}$ Andrew Lamperski  is with the Department of Electrical and Computer Engineering, University of Minnesota, Minneapolis, MN, USA 55455.
        {\tt\small alampers@umn.edu}}
}

\begin{document}

\maketitle
\thispagestyle{empty}
\pagestyle{empty}

%%%%%%%%%%%%%%%%%%%%%%%%%%%%%%%%%%%%%%%%%%%%%%%%%%%%%%%%%%%%%%%%%%%%%%%%%%%%%%%%
\begin{abstract}
This paper proposes a methodology to estimate characteristic functions of stochastic differential equations that are defined over polynomials and driven by L\'evy noise. For such systems, the time evolution of the characteristic function is governed by a partial differential equation; consequently, the stationary characteristic function can be obtained by solving an ordinary differential equation (ODE). However, except for a few special cases such as linear systems, the solution to the ODE consists of unknown coefficients. These coefficients are closely related with the stationary moments of the process, and bounds on these can be obtained by utilizing the fact that the characteristic function is positive definite. These bounds can be further used to find bounds on other higher order stationary moments and also estimate the stationary characteristic function itself. The method is finally illustrated via examples. 
\end{abstract}

%%%%%%%%%%%%%%%%%%%%%%%%%%%%%%%%%%%%%%%%%%%%%%%%%%%%%%%%%%%%%%%%%%%%%%%%%%%%
\section{Introduction}
A plethora of systems in engineering, finance, and physical sciences exhibit stochastic dynamics \cite{allen07,lande03,malliaris82,gardiner86,oksendal03}. These systems are mathematically characterized in terms of how their probability density functions (or characteristic functions) evolve with time. However, it is generally difficult to solve for these quantities other than a few special cases. One often resorts to stochastic simulations or other approximation schemes \cite{rif96, socha2008linearization, Socha05,hes05,hespanha2006modelling}. Here, we propose a method to estimate stationary characteristic functions for a class of stochastic systems. The state of this class evolves as per a stochastic differential equation driven by both white noise and L\'evy noise. For such systems, a partial differential equation (PDE) governing the time evolution of the characteristic function can be written \cite{gri04,gri13}. In stationary state, this PDE becomes an ordinary differential equation (ODE) whose order is determined by the degrees of polynomials in drift and diffusion terms. The solution to this ODE requires as many unknown coefficients as its order. Our method relies upon relating these coefficients to the stationary moments, and determining the bounds on them by exploiting the positive definiteness of the characteristic function.

Since the proposed method involves estimating some moments of a stochastic system, it is closely related with moment closure methods \cite{KuntzStatMoments16,soltani2015conditional,SinghHespanhaDM}. Specifically, recent works have shown that the positive semidefiniteness of the moments can be used to find exact lower and upper bounds on moments of a stochastic system \cite{lgs16, lgs17, gvl17, kts17}. Here, we show that such bounds can also be obtained by using the fact that the characteristic function is a positive definite function. The proposed method is superior in the sense that it links the moment bounds to approximate the stationary characteristic function and thereby the stationary probability density function. This work is also related with another category of works wherein the characteristic function is used to describe stochastic dynamical systems, including those driven by Poisson/L\'evy noise \cite{epp93, gri96, wjb99, cgk02, gri04,  dgs07, sag07, dhh09, cot11, gri13}. However, these works are typically restricted to cases where the characteristic function can be exactly solved for. Our method approximates the characteristic function beyond these.

The rest of the paper is organized as follows. Section~\ref{sec:background} provides background results on characteristic functions and their relevant properties. Section~\ref{sec:EstimatingCF} describes generators of stochastic dynamical system and proposes our method to estimate moments and characteristic function in stationary state. Section~\ref{sec:examples} illustrates the approach via examples, and finally section~\ref{sec:discussion} concludes the paper.
\paragraph*{Notation}
The set of real numbers, and natural numbers are respectively denoted by $\R$ and $\N$.
The state of a stochastic process is denoted by $X \in \R$, with a specific value taken by it denoted by $x$. 
The squared root of $-1$ is denoted by the letter $j$. The expectation operator is denoted by $\E$. Characteristic function is denoted by $\varphi(\omega)$.

\section{Mathematical Preliminaries}\label{sec:background}
In this section, we provide background results pertaining characteristic functions and their properties. The reader is referred to \cite{Fel682, ush99, biz00,sas13} for proofs and more details. For simplicity, we only consider one dimensional systems.

%\subsection{Characteristic function}
%\begin{definition}
For a univariate random variable $X$, with distribution function $F$, the characteristic function $\varphi: \R \to \C$ is defined by 
\begin{equation}
\varphi(\omega):=\E\left(e^{j\omega X}\right)=\int_{-\infty}^{\infty}e^{j\omega x} dF(x).
\end{equation}
%\end{definition}
For any random variable, the characteristic function always exists and it uniquely determines the distribution. If $F(x)$ has a density $h(x)$ then $\varphi(\omega)$ can be written in terms of $h$
\begin{equation}
\varphi(\omega) = \int_{-\infty}^{\infty}e^{j\omega x} h(x) dx.
\end{equation}
Given $\varphi(\omega)$, the distribution function $F(x)$ and/or the density $h(x)$ can be obtained via inversion. 

As an immediate consequence of its definition, the characteristic function of a random variable has the following properties
\begin{enumerate}
\item $\varphi(0)=1$.
\item $\left \lvert \varphi(\omega) \right \rvert \leq 1$.
\item $\varphi(-\omega)=\overline{\varphi(\omega)}$.
\item If the random variable has a finite $m^{th}$ order moment, then it is given by
\begin{equation}\label{eq:cfMom}
\E\left(X^m\right)=(-j)^m \frac{d^m \varphi}{d \omega^m}\lvert_{\omega=0}.
\end{equation}
\end{enumerate}
Another important property of a characteristic function that we particularly use in this work is that it is a positive definite function. 
The following theorem of Bochner forms basis of our analysis. 
\begin{theorem}[Bochner]
A continuous complex-valued function $\varphi$ on $\mathbb{R}$ is a characteristic function if and only if $\varphi(0)=1$ and $\varphi$ is positive definite \cite{ush99, biz00}.
\end{theorem}

The positive definiteness of a function required by this theorem is defined as follows.
%\begin{definition}
A complex valued function $\varphi:\mathbb{R} \to \mathbb{C}$ is said to be positive definite if the inequality
\begin{equation}
\sum_{i,j=1}^{q}\varphi(\omega_i-\omega_j)c_i\overline{c_j} \geq 0
\end{equation}
holds for every positive integer $q$, for all $\omega_1, \ldots, \omega_q \in \mathbb{R}$, and for all $c_1, \ldots, c_q \in \mathbb{C}$. 
%\end{definition}
In other words, $\varphi$ is positive definite if and only if the matrix
\small{
\begin{equation}\label{eq:M}
M=
\begin{bmatrix}
\varphi(0) & \varphi(\omega_1-\omega_2) & \ldots &  \varphi(\omega_1-\omega_q) \\
\varphi(\omega_2-\omega_1)  & \varphi(0) & \ldots &  \varphi(\omega_2-\omega_q) \\
\vdots & \vdots & \ddots & \vdots \\
\varphi(\omega_q-\omega_1)  & \varphi(\omega_q-\omega_2)  & \ldots &\varphi(0) \\
\end{bmatrix}
\succeq 0.
\end{equation}
}\normalsize
%\begin{equation}
%M = \left(\varphi(\omega_i-\omega_j)\right)_{i,j=1}^{n}
%\end{equation}
for an arbitrary choice of $q \in \mathbb{N}$ and $\omega_1, \ldots, \omega_q \in \mathbb{R}$. Here $\succeq 0$ denotes the positive semidefiniteness. 

As a consequence of this definition, a variety of properties of the characteristic function (including properties 2 and 3 above) can be established by choosing some test points and enforcing $M \succeq 0$. For example, consider the case $q=1$. Then, we must have that
\begin{equation}
\varphi(0) \geq 0.
\end{equation}
Likewise, for $q=2$, we should have
\begin{equation}
\begin{bmatrix}
\varphi(0) & \varphi(\omega_1 - \omega_2) \\
\varphi(\omega_2-\omega_1) & \varphi(0)
\end{bmatrix}
\succeq 0.
\end{equation}
Without loss of generality, we can assume $\omega_1=0$ and $\omega_2=\omega \in \mathbb{R}$. Then, we should have $\varphi(-\omega) = \overline{\varphi(\omega)}$ and $|\varphi(\omega)|^2 \leq \varphi(0)^2$.

\section{Stationary Characteristic Function of a Stochastic Process} \label{sec:EstimatingCF}
In this section, we describe the generator of a stochastic process and get a PDE for time evolution of the characteristic function. Then, we use this PDE to obtain an ODE for the stationary characteristic function and discuss its solution. 

\subsection{Stochastic dynamical system and its generator}
We consider stochastic differential equations driven by L\'evy
noise, known as \emph{It\^o-L\'evy} processes. In
this paper, we will restrict ourselves to It\^o-L\'evy processes of the form:
\begin{multline}\label{eq:JD}
  dX(t)=f(X(t))dt+g(X(t))dW(t) \\
  + \int_R y \tilde{N}(dt,dy) .
\end{multline}
Here $X \in \R$ denotes the state, $W$ is the Wiener process, 
$\tilde{N}$ is called a \emph{compensated Poisson measure}, and $f:\R \to \R$, $g:\R \to \R$ characterize the dynamics. 
In order for \eqref{eq:JD} to be well-defined we must assume that the
driving L\'evy process has finite variance. 
See \cite{oksendalapplied2009,applebaum2009levy} for more details on the formalism.

The jump process generalizes the jumps of a Poisson process. In a Poisson
process, all of the jumps have value $1$, and $N(t)$ simply counts the
number of jumps. In a L\'evy process, the jumps can take on arbitrary
values. The \emph{Poisson random measure}, $N(t,\cdot)$, is a
measure-valued stochastic process such that for a Borel set $S$, $N(t,S)$ is a Poisson
process that counts the number of jumps that took values in the set
$S$. The intensity of the Poisson random measure given by the \emph{L\'evy
  measure}, $\nu$:
\begin{equation*}
  \E[N(t,S)] = t \nu(S)
\end{equation*}
The compensated Poisson measure from \eqref{eq:JD} is the
measure-valued stochastic process defined by
$\tilde{N}(t,S) = N(t,S) - t \nu(S)$. 

Let $\psi$ be a  twice continuously differentiable function.  A
standard result in stochastic differential equations \cite{oksendalapplied2009} shows that the
dynamics of $\E[\psi(X(t))]$ are given by:

\begin{equation}
\frac{d}{dt} \E[\psi(X(t))] = \E[L\psi(X(t))]
\end{equation}
where $L\psi$ is called the \emph{generator} of the process. The
generator is defined by

\begin{multline}
\label{eq:generator}
L\psi(X)
=
\frac{\partial \psi(X)}{\partial X} 
f(X) 
+\frac{1}{2}
\frac{\partial^2 \psi(X)}{\partial X^2} 
g(X)^2
 + \\
\int_\R\left(\psi(X+y) - \psi(X)-
  \frac{\partial \psi(X)}{\partial X} y
  \right) \nu(dy).
\end{multline}

In the following, we use the generator to find a PDE that governs the
evolution of the characteristic function. We will see that for the
commonly-studied L\'evy processes, formulas exist to enable explicit
calculation of the required integral.

\subsection{Characteristic function of the process}
Consider the stochastic dynamics defined in \eqref{eq:JD}. We restrict
ourselves to the cases for which the functions $f$, and $g$ are
polynomials in the state $X$. Our goal is to compute the time
evolution of
\begin{equation}
\psi(X(t))=e^{j \omega X(t)}.
\end{equation}

The following theorem shows how a partial differential equation governing the evolution of the characteristic function can be obtained for \eqref{eq:JD}. We wish to point out that it is presented here for a formal statement, and it has been used in some form or other in several works, e.g., see \cite{gri13}.

\begin{theorem}\label{thm:ode}
Consider a one dimensional stochastic process defined in \eqref{eq:JD}. Let $f$ and $g^2$ be finite polynomials of degrees $d_f \in \N$ and $d_{g} \in \N$ respectively. Assuming that a stationary distribution exists, the characteristic function of the stationary distribution satisfies the following ordinary differential equation of order $n=\max \{d_f, d_g\}$:
\begin{equation}\label{eq:StODE}
j \omega  \sum_{l=0}^{d_f}a_{f_l} j^{-l} \frac{\partial^l \varphi}{\partial \omega^l} +\frac{1}{2} (j \omega)^2  \sum_{l=0}^{d_g}a_{g_l}  j^{-l} \frac{\partial^l \varphi}{\partial \omega^l}+\eta(\omega)\varphi=0,
\end{equation}
where $\eta(\omega)= \int_\R\left(e^{j\omega y} - 1 -
  j \omega y
  \right) \nu(dy)$.
\end{theorem}
\begin{proof} 
Since we assume that $f$, and $g^2$ are polynomials, without loss of generality we can take their forms to be
\begin{equation}
f(X)=\sum_{l=0}^{d_f}a_{f_l} X^{l}, \quad g^2=\sum_{l=0}^{d_g}a_{g_l} X^{l},
\end{equation}
where $a_{f_l} \in \R$, and $a_{g_l} \in \R$ are coefficients. Taking $\psi(X)=e^{j\omega X},\; \omega \in \R$, we have that
\begin{align}\label{eq:derivativesPsi}
\frac{\partial^m \psi}{\partial X^m}=(j\omega)^m\psi, \quad \frac{\partial^m\psi}{\partial \omega^m}=(j X)^m\psi \quad \forall m \in \N
\end{align}

Using these, we can write $L\psi(X)$ as
\begin{multline}
L\psi(X)=j \omega   \sum_{l=0}^{d_f}a_{f_l} X^{l} +\frac{1}{2} (j \omega)^2  \sum_{l=0}^{d_g}a_{g_l} X^{l}
+ \\
\psi(X) \int_\R\left(e^{j\omega y} - 1 -
  j \omega y
  \right) \nu(dy).
\end{multline}
Using \eqref{eq:derivativesPsi}, the terms $X^l \psi$ can be replaced by $(j)^{-l}\frac{\partial^l \psi}{\partial \omega^l}$. This yields
\begin{multline}
L\psi(X)=j \omega  \sum_{l=0}^{d_f}a_{f_l} j^{-l} \frac{\partial^l \psi}{\partial \omega^l} +\frac{1}{2} (j \omega)^2  \sum_{l=0}^{d_g}a_{g_l}  j^{-l} \frac{\partial^l \psi}{\partial \omega^l} 
+ \\
\psi(X) \int_\R\left(e^{j\omega y} - 1 -
  j \omega y
  \right) \nu(dy).
\end{multline}

Taking expectation, we get a partial differential equation in characteristic function $\varphi=\E(\psi(X))$
\begin{multline}\label{eq:pde}
\frac{\partial \varphi}{\partial t}=j \omega  \sum_{l=0}^{d_f}a_{f_l} j^{-l} \frac{\partial^l \varphi}{\partial \omega^l} 
+\\
\frac{1}{2} (j \omega)^2  \sum_{l=0}^{d_g}a_{g_l}  j^{-l} \frac{\partial^l \varphi}{\partial \omega^l} 
+
\eta(\omega) \varphi,
\end{multline}
where we used $\eta(\omega)$ to denote the integral $\int_\R\left(e^{j\omega y} - 1 - j \omega y \right) \nu(dy)$.
If the stationary distribution exists (see \cite{met12} for details), then we must have that $\frac{\partial \varphi}{\partial t}=0$. This results in the ordinary differential equation \eqref{eq:StODE}. 
%that case, $\varphi$ is given by the solution of  the following ordinary differential equation
%\begin{equation}
%j \omega  \sum_{l=0}^{d_f}a_{f_l} j^{-l} \frac{\partial^l \varphi}{\partial \omega^l} +\frac{1}{2} (j \omega)^2  \sum_{l=0}^{d_g}a_{g_l}  j^{-l} \frac{\partial^l \varphi}{\partial \omega^l}+\eta(\omega)\varphi=0.
%\end{equation}
The degree of this ODE is $n=\max \{d_f, d_g \}$.
\end{proof}
\begin{remark}
  The function $\eta(\omega)$ is known in as the \emph{characteristic
    exponent}, \cite{applebaum2009levy}. A special case of the
  L\'evy-Khintchine formula shows that the compensated Poisson process
  has characteristic function given by:
\begin{equation}
  \E\left[
    e^{j\omega \int_0^ty \tilde{N}(dt,dy)}
    \right] = e^{t\eta(\omega)}. 
  \end{equation}
  Thus, the statistics of the jump process $\int_0^t y
  \tilde{N}(dt,dy)$ are entirely determined by $\eta(\omega)$.
  
  For commonly-studied L\'evy processes, the characteristic exponent
  has an explicit formula. For example, the gamma process $\tau(t)$ has L\'evy
  measure 
  \begin{equation}
    \nu(dy) = ay^{-1} e^{-by} dy, 
  \end{equation}
  where $a$ and $b$ are positive parameters. It can be shown that
  \begin{equation}
    \E[e^{j\omega \tau(t)}] =
    e^{-ta\log\left(1-j\frac{\omega}{b}\right)}
    \quad \textrm{ and } \quad \E[\tau(t)] = t\frac{a}{b}
  \end{equation}
  It follows that the compensated Gamma process, $\tau(t)-\E[\tau(t)]$,
  has characteristic exponent given by
  $-a\log\left(1-j\frac{\omega}{b}\right)-j\omega a/b$.

  A more complex example, which will be studied in an example below is
  the variance-gamma process: $\sigma W(\tau(t))$. This process is
  formed by composing a standard Brownian motion $W(t)$ with a gamma
  process. In this case, it can be shown that
  \begin{equation}
    \E[e^{j\omega W(\tau(t))}] = e^{-ta\log\left(
        1+\frac{(\sigma \omega)^2}{2b}
        \right)} \quad \textrm{and} \quad \E[W(\tau(t))] = 0.
  \end{equation}
  It follows that compensated variance-gamma process is simply the
  variance-gamma process. Furthermore, the characteristic exponent is
  given by
  \begin{equation}
    \label{eq:vgExp}
    \eta(\omega) = - a \log\left(
1+\frac{(\sigma\omega)^2}{2b}
\right),
  \end{equation}
  
\end{remark}
Next, we discuss the solution of the ODE for the stationary characteristic function.

\subsection{Solving for stationary characteristic function}
The ODE in \eqref{eq:StODE} cannot be solved analytically except for a  handful of cases. However, it can be solved via numerical techniques. Either way it would require $n$ initial/intermediate/boundary values to find the solution. Other than the usual $\varphi(0)=1$, previous works have either utilized prior knowledge about the system (e.g., the distribution is symmetric), or used $\lim_{|\omega|\to \infty}\varphi(\omega) = 0$ and $\lim_{|\omega|\to \infty}\frac{\partial^l \varphi(\omega)}{\partial \omega^l} = 0$ for some $l$ \cite{gri13}. In practice these are hard to incorporate in a solution. Furthermore, if one is interested only in stationary moments, then solution of $\varphi(\omega)$ only in neighborhood of zero is sufficient. 

We propose a different approach to compute both the moments and the characteristic function. This approach relies on two ideas. First being the fact that the characteristic function is related with the moments as
\begin{equation}\label{eq:momMu}
 \frac{\partial^l \varphi}{\partial \omega^l}\left |_{\omega=0} \right. = j^{l} \mu_l, \; l=\{1, \ldots, n-1\},
 \end{equation}
where $\mu_l \in \R$ represents the $l^{th}$ order moment. Thus, the moments are natural quantities to be used in computing the characteristic function. Second idea is to utilize the Bochner's theorem to estimate the moments $\mu_l$. In particular, we can use $\varphi(0)=1$ and positive semidefinite property of the matrix $M$ in \eqref{eq:M}. Using these, a semidefinite program can be formulated that gives lower and upper bounds on $\mu_l$ as stated in the theorem below.

\begin{theorem}
 Consider the one dimensional system defined by \eqref{eq:JD}. Assuming that $f$ and $g^2$ are polynomials of degree $d_f$ and $d_g$, a lower bound on a moment $\mu_l$ can be obtained via the semidefinite program
\begin{subequations}\label{eq:SDPsetup}
\begin{align}
& \min \;  \mu_k  \label{eq:SDPsetupOBJ}\\
& j \omega  \sum_{l=0}^{d_f}a_{f_l} j^{-l} \frac{\partial^l \varphi}{\partial \omega^l} +\frac{1}{2} (j \omega)^2  \sum_{l=0}^{d_g}a_{g_l}  j^{-l} \frac{\partial^l \varphi}{\partial \omega^l} + \eta(\omega) \varphi=0, \label{eq:SDPsetupODE}\\
& \frac{\partial^l \varphi}{\partial \omega^l}\left |_{\omega=0} \right. = j^{l} \mu_l, \; \; l={1, 2, \ldots n}, \label{eq:SDPsetupMom}\\
& \varphi(0)=1, \label{eq:SDPsetupNorm}\\
& M \succeq 0.\label{eq:SDPsetupM}
\end{align}
\end{subequations}
Here $k=\{1, \ldots, n\}$ with $n=\max \{d_f, d_g\}$ and $M$ is defined as in \eqref{eq:M}. Further, the minimum value obtained by the program increases as size of $M$ is increased by including more test points.
\end{theorem}
\begin{proof}
Since $f$ and $g^2$ are assumed to be polynomials, Theorem~\ref{thm:ode} implies that the characteristic function $\varphi(\omega)$ satisfies the ODE of order $n$ given by \eqref{eq:SDPsetupODE}. The moments are related with the derivatives of the characteristic function by the linear constraints in \eqref{eq:SDPsetupMom}. The constraint in \eqref{eq:SDPsetupNorm} and \eqref{eq:SDPsetupM} are a consequence of the Bochner's theorem. Since the objective function is linear in decision variables $\mu_k$ and the constraints are either equality or semidefinite constraints, the optimization problem is a semidefinite program \cite{boyd2004convex}.

Now suppose that the size of $M$ is increased by including more test points $\omega_1, \ldots, \omega_q$. This corresponds to adding more constraints in the program, and the solution cannot get worse by doing so.
\end{proof}

The upper bound on $\mu_k$ can be found by minimizing $-\mu_k$. Note that we can choose any test points $\omega_1, \omega_2, \ldots, \omega_q \in \mathbb{R}$ in order to generate the matrix $M$. For sake of simplicity, we will choose uniformly spaced values on the real-line. 
%Future work will explore how picking non-uniformly spaced values affects the performance of the program.

The above semidefinite program can be used to compute lower and upper bounds on each of the moments $\mu_k$. These values can be then used to determine the solution to the ODE for characteristic function and thereby finding an approximation of the characteristic function. If the lower and upper bounds on each of the moments are reasonably close, then the approximate characteristic function would be quite close to the true characteristic function. It can be further used to compute the stationary probability density via inversion. 

\begin{remark}
  \label{rem:OptControl}
The formulation in \eqref{eq:SDPsetup} can be interpreted as an optimal control problem for a linear time varying system. Specifically, consider a state vector $z=\begin{bmatrix} \varphi & \frac{\partial \varphi}{\partial \omega} & \cdots & \frac{\partial^{n-1} \varphi}{\partial \omega^{n-1}}\end{bmatrix}^\top$. Then, the differential equation describing the characteristic function becomes a linear time (in $\omega$-space) varying system
\begin{equation}
\frac{dz}{d\omega}=\mathcal{P}(\omega) z(\omega),
\end{equation}
for an appropriately defined matrix $\mathcal{P}$. In this setup, the objective would be to optimize the elements of $z(0)$ subject to the linear matrix inequality \eqref{eq:SDPsetupM}. Note that the first element of $z(0)$ is given by $1$.

\end{remark}

%\subsection{Computing stationary moments of higher order}
\begin{remark}
The semidefinite program in \eqref{eq:SDPsetup} can be used to find bounds on first $n$ moments where $n$ is the degree of the ODE \eqref{eq:SDPsetupODE}. If one is interested in computing the higher order moments, the approximate characteristic function can be differentiated and computed at $\omega=0$. By doing so, bounds on the higher order moments can also be computed. Alternatively, for systems with finite moments, one can compute the bounds on first $n$ moments via the proposed method, and then use the fact that stationary moments are related via a linear system of equations given by \cite{lgs17}
\begin{equation}\label{eq:XXbar}
A \mathcal{X} + B \overline{\mathcal{X}}=0.
\end{equation}
Here $\mathcal{X}$ is collection of moments up to some order and $\overline{\mathcal{X}}$ contains moments of order higher than those in $\mathcal{X}$. The number of elements of $\overline{\mathcal{X}}$ is as many as the degree of nonlinearity in the system given by $n$. While the usual moment closure methods estimate elements of $\overline{\mathcal{X}}$ in terms of those of $\mathcal{X}$, we simply supplement \eqref{eq:XXbar} with lower and upper bounds on $n$ moments and thereby compute bounds on all other moments in \eqref{eq:XXbar}. 
\end{remark}
In the next section, we illustrate the proposed method on some simple examples and verify its performance.

\section{Examples}\label{sec:examples}
\begin{example}[Stochastic Logistic Model]
\begin{figure}
\centering
\includegraphics[width=0.95\linewidth]{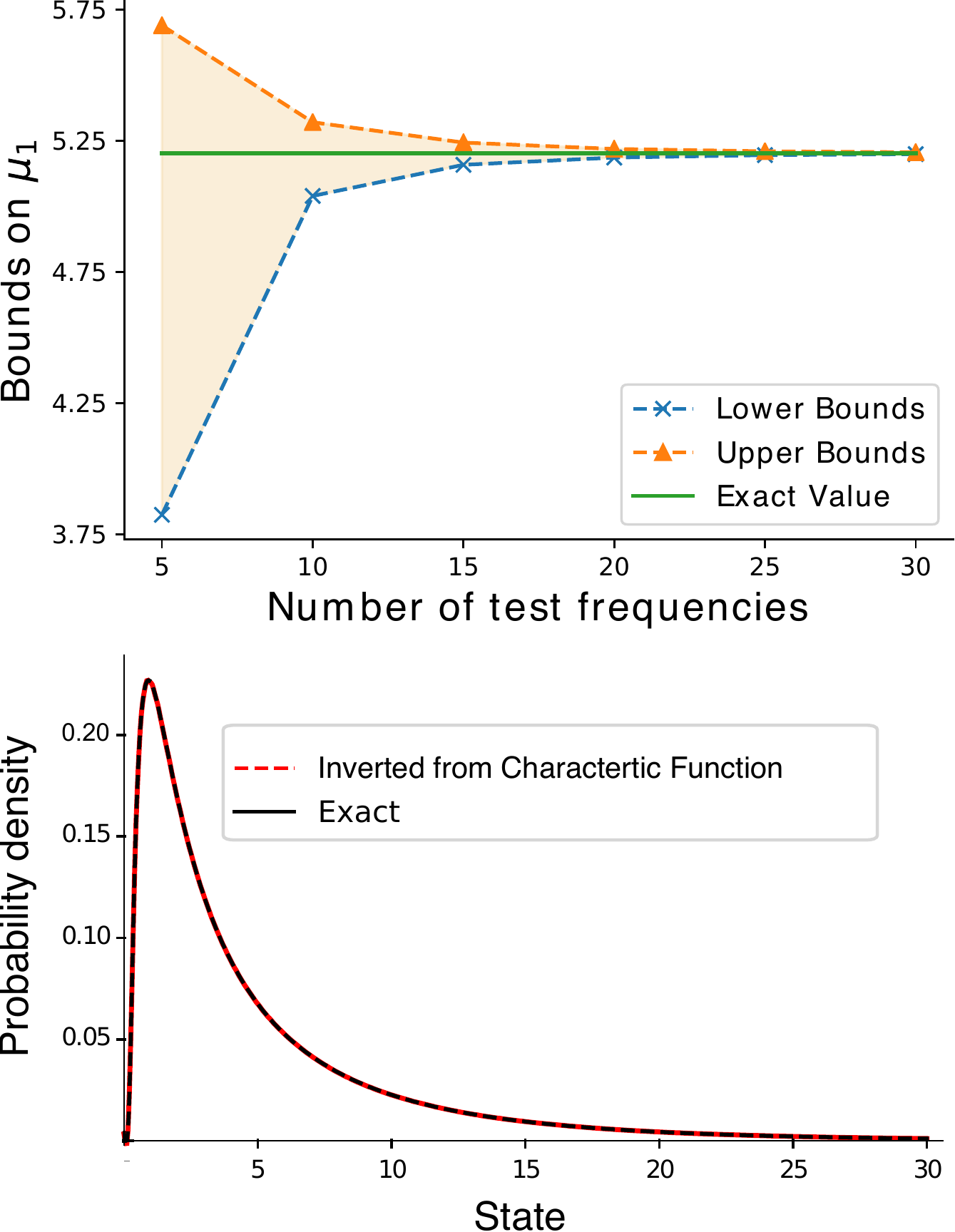}\\
\caption{Comparison of estimated mean and characteristic function with exact values. {\it Top}: The mean $\mu_1$ is estimated via semidefinite program in \eqref{eq:SDPsetup}. The matrix $M$ is constructed by taking frequencies spaced by $one$ (e.g., $\omega_1=1, \omega_2=2, \ldots$). Increasing the number of test frequencies leads to a better estimation of the moment as both lower and upper bounds converge.Estimation of mean via SDP. {\it Bottom}: The characteristic function is estimated by using the middle value of lower and upper bounds for $\mu_1$ obtained from $30$ test frequencies in \eqref{eq:phi}--\eqref{eqn:cfCoeff}. The estimated probability densify function is obtained by numerically inverting the characteristic function. Notably, this particular system is exactly solvable, and the exact probability densify function is provided for comparison purpose, showing an excellent match with its estimated counterpart.}
\label{fig:mu1Logistic}
\end{figure}

Consider the following modified stochastic logistic growth model
\begin{equation}
dX = (1+X-0.1 X^2)dt + \sqrt{2} X dW.
\end{equation}
Without the constant term in the drift, this model is widely-used in modeling population growth \cite{allen07}. We added the constant term so that the trivial solution $X=0$ is ruled out. 

Using \eqref{eq:pde}, the characteristic function evolves as per
\begin{align}
\frac{\partial \varphi}{\partial t}
=
j \omega \varphi + \omega \frac{\partial \varphi}{\partial \omega}+j 0.1 \omega \frac{\partial^2 \varphi}{\partial \omega^2} + \omega^2 \frac{\partial^2 \varphi}{\partial \omega^2}.
\end{align}
Therefore, the stationary characteristic function is the solution to the following differential equation
\begin{align}
j \omega \varphi + \omega \frac{\partial \varphi}{\partial \omega}+\left(j 0.1 \omega  + \omega^2\right) \frac{\partial^2 \varphi}{\partial \omega^2} =0.
\end{align}

It can be shown that the above ODE has the following generalized solution 
\begin{multline}\label{eq:phi}
\varphi (\omega )
= 
\frac{c_1}{\sqrt{5}} I_{0}\left(2 \sqrt{0.1-j \omega} \right)+\frac{c_2}{\sqrt{5}} K_{0}\left(2 \sqrt{0.1-j \omega}\right),
\end{multline}
where $I$ and $K$ denote the modified Bessel functions of first and second kinds, and $c_1$, $c_2$ are unknown coefficients. As expected,  the number of unknown coefficients is same as the order of nonlinearity in the dynamics.
\end{example}

To determine the coefficients, we can use the fact that $\varphi(0)=1$ and $\frac{\partial \varphi}{\partial \omega}\left |_{\omega=0} \right.=j \mu_1$, where $\mu_1 \in \R$ is the mean that is to be determined. This results in 
\begin{subequations}\label{eqn:cfCoeff}
\begin{align}
& \frac{c_1}{\sqrt{5}} I_{0}\left(2 \sqrt{0.1} \right)+\frac{c_2}{\sqrt{5}} K_{0}\left(2 \sqrt{0.1}\right)=1 \\
& -j \sqrt{2} c_1 I_1\left(2\sqrt{0.1}\right)+j \sqrt{2} c_2 K_1\left(2\sqrt{0.1}\right)=j \mu_1
\end{align}
\end{subequations}
Using these, the stationary characteristic function can be written in terms of only one unknown $\mu_1$. Now, we can compute bounds on $\mu_1$ using the semidefinite program as in \eqref{eq:SDPsetup}. By choosing uniformly spaced values of $\omega_1, \omega_2, \ldots, \omega_n$, we computed the maximum and minimum allowable values $\mu_1$. We also find that increasing the size of the program by choosing more test points improves both lower and upper bounds. Taking $30$ test points at $\omega_1=1, \ldots \omega_{30}=30$, we get $5.2024 \leq \mu_1 \leq 5.2025$ (see Fig.~\ref{fig:mu1Logistic}, Top). This result is in excellent agreement with Monte Carlo simulations which yield an value of $5.2$ for $10000$ simulations.

Using the value of $\mu_1$ obtained here, we can use the characteristic function to reconstruct the probability density function of the stationary distribution (see Fig.~\ref{fig:mu1Logistic}, Bottom). As mentioned earlier, the bounds on $\mu_1$ can be used to estimate bounds on higher order moments as well (results not shown here).

\begin{example}[Variance Gamma Process]
\begin{figure}
\includegraphics[width=\linewidth]{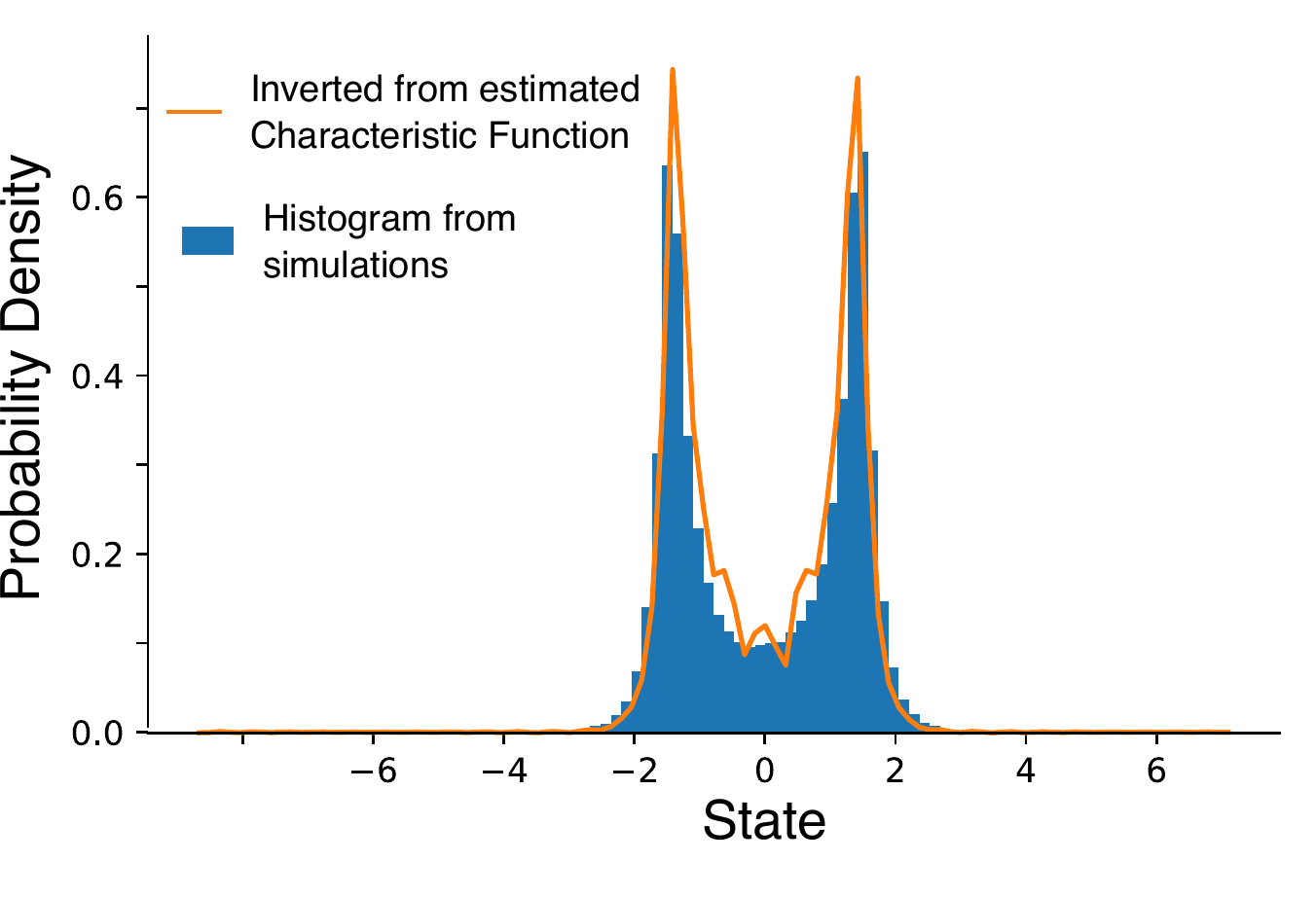}
\caption{Comparison of probability density functions obtained from the proposed method and simulations. The histogram is obtained via $1000$ simulations of the process, whereas the estimated probability density is inverted from the estimated characteristic function by solving the semidefinite program in \eqref{eq:SDPsetup}. The optimization was performed by casting the problem as an optimal control problem described in Remark 2, with evenly spaced $80$ test frequencies.}
\label{fig:levy}
\end{figure}

Consider the following process
\begin{equation}
  dX=(-5X^3+10X)dt+dW(\tau(t)),
\end{equation}
where $W(\tau(t))$ is the variance-gamma process described
above. Recall, in particular, $W(\tau(t))$ has zero mean, and so the
compensated variance-gamma process is the same as the original
variance-gamma process. Furthermore, recall the expression for the
characteristic exponent, $\eta(\omega)$, from \eqref{eq:vgExp}.

Using \eqref{eq:pde}, the characteristic function of this process is governed by
\begin{equation}
\frac{\partial \varphi}{\partial t}=5 \omega^3 \frac{\partial^3 \varphi}{\partial \omega^3}+10 \omega \frac{\partial \varphi}{\partial \omega}+\eta(\omega)\varphi
\end{equation}
As before, this ODE solution will have three unknown coefficients, out
of which two could be readily computed by using the facts that
$\varphi(0)=1$ and the stationary mean value of $X$ is zero due to
symmetry of the process. Numerically solving the optimization problem
that minimizes $\mu_2$ (the second moment), and inverting the solution
yields quite accurate estimate of the stationary probability density
function (see Fig.~\ref{fig:levy}).

The optimization was performed by discretizing the optimal control
problem described in Remark~\ref{rem:OptControl} via the trapezoidal
rule, and then directly solving the corresponding SDP. Here, the
decision variables were the state variables $z(\omega_i)$, where
$\omega_i$ are the discretization frequencies. One subtlety is that in
order to express the matrix $M$ in terms of decision variables, the
frequencies, $\omega_i$, must be evenly spaced.   
\end{example}

\section{Conclusion}\label{sec:discussion}
In this paper, we proposed a method to estimate stationary characteristic function for a class of stochastic dynamical system driven by both white and L\'evy noise.  The characteristic function for these stochastic systems is governed by an ODE. The method relies upon restricting the solutions of the ODE to positive definite functions and casts the problem as a semidefinite program. In future work, we would extend this framework to multidimensional cases. It would also be interesting to explore whether transient solutions of the characteristic function (described via PDE) can be obtained via a similar method.

\section*{Acknowledgment}
AS is supported by the National Science Foundation Grant ECCS-1711548.

\bibliography{ECC18Ref}

\end{document}